\documentclass[preprint,11pt,sort&compress,numbers]{elsarticle}

\usepackage{amsmath,amsthm,amsfonts,amssymb,latexsym,mathrsfs,color,cases,url,enumerate}

\newtheorem{thm}{Theorem}%[section]
\theoremstyle{definition}

\newtheorem{rem}[thm]{Remark}
\newtheorem{exm}[thm]{Example}

\def\R{\mathcal{R}}
\def\L{\mathfrak{L}}%\mathcal{L}}

\usepackage{geometry}
\journal{LAA}

\begin{document}

\begin{frontmatter}

\title{Yet another criterion for the total positivity of Riordan arrays}
\author[a]{Jianxi Mao\corref{cor1}}
\ead{maojianxi@hotmail.com}
\author[b]{Lili Mu\corref{cor2}}
\ead{lilimu@jsnu.edu.cn}
\cortext[cor2]{Corresponding author.}
\author[a]{Yi Wang\corref{cor3}}
\ead{wangyi@dlut.edu.cn}
\address[a]{School of Mathematical Sciences, Dalian University of Technology, Dalian 116024, P.R. China}
\address[b]{School of Mathematics and Statistics, Jiangsu Normal University, Xuzhou 221116, P.R. China}

\begin{abstract}
%We establish a new criterion for the total positivity of Riordan arrays.
Let $R=\R(d(t),h(t))$ be a Riordan array,
where $d(t)=\sum_{n\ge 0}d_nt^n$ and $h(t)=\sum_{n\ge 0}h_nt^n$.
We show that if the matrix
\begin{equation*}
\left[\begin{array}{ccccc}
d_0 & h_0 & 0 & 0 &\cdots\\
d_1 & h_1 & h_0 & 0 &\\
d_2 & h_2 & h_1 & h_0 &\\
\vdots&\vdots&&&\ddots
\end{array}\right]
\end{equation*}
is totally positive,
then so is the Riordan array $R$.
\end{abstract}


\end{frontmatter}

%\section{Introduction}

Following Karlin~\cite{Kar68},
an infinite matrix is called {\it totally positive} (or shortly, TP),
if its minors of all orders are nonnegative.
An infinite nonnegative sequence $(a_n)_{n\ge0}$ is called a {\it P\'olya frequency sequence}
(or shortly, PF), if its Toeplitz matrix
$$
[a_{i-j}]_{i,j\ge0}=
\left[
    \begin{array}{ccccc}
        a_0 \\
        a_1    & a_0 \\
        a_2    & a_1 & a_0 \\
        a_3    & a_2 & a_1 & a_0 \\
        \vdots &     &     &     & \ddots \\
    \end{array}
\right]
$$
is TP.
We say that a finite sequence $a_0,a_1,\ldots,a_n$ is PF
if the corresponding infinite sequence $a_0,a_1,\ldots,a_n,0,\ldots$ is PF.
A fundamental characterization for PF sequences is due to Schoenberg and Edrei,
which states that
a sequence $(a_n)_{n\ge 0}$ is PF if and only if its generating function
\begin{equation*}\label{pf-c}
\sum_{n\ge 0}a_nt^n=Ct^ke^{\gamma t} \frac{\prod_{j\ge0} (1+\alpha_j t)}{\prod_{j\ge0} (1-\beta_j t)},
\end{equation*}
where $C>0, k\in\mathbb{N}, \alpha_j,\beta_j,\gamma\ge 0$, and $\sum_{j\ge0} (\alpha_j+\beta_j)<+\infty$
(see~\cite[p. 412]{Kar68} for instance).
In this case, the generating function is called a {\it P\'olya frequency formal power series}. %~\cite{Bre89}.
We refer the reader to~\cite{And87,FJ11,Pin10,Slo20,WY18} for the total positivity of matrices.
Our concern in this note is the total positivity of Riordan arrays.

Riordan arrays play an important unifying role in enumerative combinatorics
\cite{SGWW91,Spr94}.
Let $d(t)=\sum_{n\ge 0}d_nt^n$ and $h(t)=\sum_{n\ge 0}h_nt^n$
be two formal power series.
A {\it Riordan array}, denoted by $\R(d(t),h(t))$,
is an infinite matrix
whose generating function of the $k$th column is $d(t)h^k(t)$ for $k\ge 0$.
Chen and Wang~\cite[Theorem 2.1]{CW19} gave the following criterion for the total positivity of Riordan arrays.

\begin{thm}[{\cite[Theorem 2.1]{CW19}}]\label{cw-c}
Let $R=(d(t),h(t))$ be a Riordan array.
If both $d(t)$ and $h(t)$ are P\'olya frequency formal power series,
then $R$ is totally positive.
\end{thm}

We say that $\R(d(t),h(t))$ is {\it proper}
if $d_0\neq 0, h_0=0$ and $h_1\neq 0$.
In this case,
$\R(d(t),h(t))$ is an infinite lower triangular matrix.
%Proper Riordan arrays form a group under the matrix multiplication.
It is well known that a proper Riordan array $R=[r_{n,k}]_{n,k\ge 0}$
can be characterized by two sequences
$(a_n)_{n\ge 0}$ and $(z_n)_{n\ge 0}$ such that
\begin{equation*}\label{rrr-c}
r_{0,0}=1,\quad r_{n+1,0}=\sum_{j\ge 0}z_jr_{n,j},\quad r_{n+1,k+1}=\sum_{j\ge 0}a_jr_{n,k+j}
\end{equation*}
for $n,k\ge 0$ (see \cite{,He15, HS09} for instance).
Call $(a_n)_{n\ge 0}$ and $(z_n)_{n\ge 0}$
the $A$- and $Z$-sequences of $R$ respectively.
Chen {\it et al.}~\cite[Theorem 2.1 (i)]{CLW15} gave the following criterion for the total positivity of Riordan arrays.

\begin{thm}[{\cite[Theorem 2.1 (i)]{CLW15}}]\label{clw-c}
Let $R$ be the proper Riordan array with the $A$- and $Z$-sequences $(a_n)_{n\ge 0}$ and $(z_n)_{n\ge 0}$.
If the product matrix
 \begin{equation*}\label{ra-pm}
P=\left[
\begin{array}{ccccc}
z_0 & a_0 & 0 & 0 &\cdots\\
z_1 & a_1 & a_0 & 0 &\\
z_2 & a_2 & a_1 & a_0 &\\
%z_3 & a_3 & a_2 & a_1 & \ddots\\
\vdots&\vdots&&&\ddots
\end{array}
\right]
\end{equation*}
is totally positive,
then so is $R$.
\end{thm}

In this note we establish a new criterion for the total positivity of Riordan arrays,
which can be viewed as a dual version of Theorem \ref{clw-c} in a certain sense.
%and then apply it to show that the Lucas matrix is TP,
%which can be obtained neither by Theorem~\ref{cw-c}, as well as nor by Theorem~\ref{clw-c}.

\begin{thm}\label{mmw-c}
Let $R=(d(t),h(t))$ be a Riordan array,
where $d(t)=\sum_{n\ge 0}d_nt^n$ and $h(t)=\sum_{n\ge 0}h_nt^n$.
If the Hessenberg matrix
\begin{equation}\label{dh-m}
H=\left[\begin{array}{ccccc}
d_0 & h_0 & 0 & 0 &\cdots\\
d_1 & h_1 & h_0 & 0 &\\
d_2 & h_2 & h_1 & h_0 &\\
\vdots&\vdots&&&\ddots
\end{array}\right]
\end{equation}
is totally positive,
then so is $R$.
\end{thm}

%In the next section, we first give a proof of Theorem \ref{mmw-c},
%then apply Theorem \ref{mmw-c} to show that the Lucas matrix is TP,
%which can be obtained neither by Theorem \ref{cw-c}, as well as nor by Theorem \ref{clw-c}.
%We also point out that Theorem \ref{mmw-c} implies Theorem \ref{cw-c}.

%\section{Proof and applications of Theorem \ref{mmw-c}}

%We first give a proof of Theorem \ref{mmw-c} and then present its applications.

\begin{proof}%[Proof of Theorem \ref{mmw-c}]
Let $R[n]$ be the submatrix consisting of the first $n+1$ columns of $R$.
Clearly, $R$ is TP if and only if all submatrices $R[n]$ are TP for $n\ge 0$.
So it suffices to show that $R[n]$ is TP for all $n\ge 0$.
We proceed by induction on $n$.

Let $R=[r_{n,k}]_{n,k\ge 0}$. %and let $R_n$ denote the $(n+1)$th column of $R$.
Since the generating function of the $k$th column of $R$ is $d(x)h^k(x)$,
we have
$$
\left[\begin{array}{c}
r_{0,0} \\
r_{1,0} \\
r_{2,0} \\
\vdots \\
\end{array}\right]=
\left[\begin{array}{c}
d_0 \\
d_1 \\
d_2 \\
\vdots \\
\end{array}\right],\qquad
\left[\begin{array}{c}
r_{0,k} \\
r_{1,k} \\
r_{2,k} \\
\vdots \\
\end{array}\right]=
\left[\begin{array}{ccccc}
h_0 & 0 & 0 &\cdots\\
h_1 & h_0 & 0 &\\
h_2 & h_1 & h_0 &\\
\vdots&&&\ddots
\end{array}\right]
\left[\begin{array}{c}
r_{0,k-1} \\
r_{1,k-1} \\
r_{2,k-1} \\
\vdots \\
\end{array}\right]
$$
for $k\ge 1$.
It follows that
$$
\left[\begin{array}{cccc}
r_{0,0} & r_{0,1} & \cdots & r_{0,n+1} \\
r_{1,0} & r_{1,1} & \cdots & r_{1,n+1} \\
r_{2,0} & r_{2,1} & \cdots & r_{2,n+1} \\
\vdots &  &  & \vdots \\
\end{array}\right]=
\left[\begin{array}{ccccc}
d_0 & h_0 & 0 & 0 &\cdots\\
d_1 & h_1 & h_0 & 0 &\\
d_2 & h_2 & h_1 & h_0 &\\
\vdots&\vdots&&&\ddots
\end{array}\right]
\left[\begin{array}{ccccc}
1 & 0 & 0 & \cdots & 0 \\
0 & r_{0,0} & r_{0,1} & \cdots & r_{0,n} \\
0 & r_{1,0} & r_{1,1} & \cdots & r_{1,n} \\
\vdots & \vdots & & & \vdots \\
\end{array}\right],
$$
or equivalently,
\begin{equation}\label{h-p}
R[n+1]=H
\left[\begin{array}{cc}
1 & 0 \\
0 & R[n] \\
\end{array}\right].
\end{equation}
The first matrix $H$ on the right-hand side of \eqref{h-p} is TP by assumption,
which implies that all $d_n$ are nonnegative,
and the matrix $R[0]$ is therefore TP.
Assume now that the matrix $R[n]$ is TP for $n\ge 0$. %by induction hypothesis,
Then the second matrix
$\left[\begin{array}{cc}
1 & 0 \\
0 & R[n] \\
\end{array}\right]$
on the right-hand side of \eqref{h-p} is also TP.
It is well known that the product of TP matrices is still TP
by the classic Cauchy-Binet formula.
%(see \cite[]{Pin} for instance)
Thus the matrix $R[n+1]$ on the left-hand side of \eqref{h-p} is TP.
The matrix $R$ is therefore TP by induction,
and the proof is complete.
\end{proof}

\begin{exm}
%Lucas matrix (A034807).
Consider Lucas polynomials $L_n(x)=\sum_kL_{n,k}x^k$ defined by
\begin{equation}\label{lp-r}
L_{n+1}(x)=L_n(x)+xL_{n-1}(x)
\end{equation}
with $L_0(x)=2$ and $L_1(x)=1$.
Lucas matrix is the lower triangular infinite matrix
$$L=[L_{n,k}]_{n,k\ge 0}=
\left[\begin{array}{ccccc}
    2 &  &  &  &  \\
    1 &  &  &  &  \\
    1 & 2 &  &  &  \\
    1 & 3 &  &  &  \\
    1 & 4 & 2 &  &  \\
    1 & 5 & 5 &  &  \\
    1& 6 & 9 & 2& \\
    1& 7 & 14 & 7& \\
    \vdots&&&&\ddots
  \end{array}\right].
$$
Let $\L_k(t)=\sum_{n\ge 0}L_{n,k}t^n$ denote the generating function of the $k$th column of $L$ for $k\ge 0$.
Clearly, $\L_0(t)=(2-t)/(1-t)$.
On the other hand,
we have $L_{n,k}=L_{n-1,k}+L_{n-2,k-1}$ for $n>k>0$ by \eqref{lp-r}.
It follows that $\L_n(t)=\frac{t^2}{1-t}\L_{n-1}(t)$ for $n\ge 1$.
%where $L_{0,0}=2, L_{n,0}=1$ and
%$$L_{n,k}=\binom{n-k}{k}+\binom{n-k-1}{k-1}=\frac{n}{k}\binom{n-k-1}{k-1}$$ for $n\ge k>0$
%and $L_{n,k}=L_{n-1,k}+L_{n-2,k-1}$.
Thus $L$ is a Riordan array:
$$L=\R\left(\frac{2-t}{1-t},\frac{t^2}{1-t}\right).$$
The corresponding Hessenberg matrix is
$$H=\left[\begin{array}{ccccc}
2 & 0 & 0 & 0 &\cdots\\
1 & 0 & 0 & 0 &\\
1 & 1 & 0 & 0 &\\
1 & 1 & 1 & 0 &\\
1 & 1 & 1 & 1 &\\
\vdots&\vdots&&&\ddots
\end{array}\right],$$
which is clearly TP, and so is $L$ by Theorem \ref{mmw-c}.

However, the total positivity of $L$ can be followed neither from Theorem \ref{cw-c}
since $d(t)=(2-t)/(1-t)$ is not PF,
nor from Theorem \ref{clw-c} since $L$ is improper.
\end{exm}

\begin{rem}
We can show that Theorem \ref{mmw-c} implies Theorem \ref{cw-c}.
Consider first several important classes of proper Riordan arrays $R=\R(d(t),h(t))$.
\begin{enumerate}[\rm (i)]
\item %Toeplitz-type: $h(t)=t$.
Let $h(t)=t$. Then $R$ is a Toeplitz-type Riordan array,
which is precisely the Toeplitz matrix of the sequence $(d_n)_{n\ge 0}$.
If $d(t)$ is PF, then $\R(d(t),t)$ is TP.
\item %Bell-type: $h(t)=td(t)$.
Let $h(t)=td(t)$. Then $R$ is a Bell-type Riordan array.
%A Bell-type Riordan array $\R(d(t),td(t))$
In this case,
the corresponding Hessenberg matrix \eqref{dh-m} is
the Toeplitz matrix of $(d_n)_{n\ge 0}$.
If $d(t)$ is PF, i.e., $h(t)$ is PF,
then $\R(h(t)/t,h(t))$ is TP by Theorem \ref{mmw-c}.
\item %Lagrange-type: $d(t)=1$.
Let $d(t)=1$. Then $R$ is a Lagrange-type Riordan array.
Note that
\begin{equation*}\label{lb-r}
\R(1,h(t))=
\left[\begin{array}{cc}
1 & 0 \\
0 & \R\left(h(t)/t,h(t)\right) \\
\end{array}\right].
\end{equation*}
If $h(t)$ is PF, then $\R(h(t)/t,h(t))$ is TP,
and so is $\R(1,h(t))$.
\end{enumerate}

It is well known \cite{SGWW91}
that every proper Riordan array can be decomposed into
the product of a Toeplitz-type Riordan array
and a Lagrange-type Riordan array:
\begin{equation*}\label{tl-p}
\R(d(t),h(t))=\R(d(t),t)\cdot\R(1,h(t)).
\end{equation*}
We conclude that if both $d(t)$ and $h(t)$ are PF,
then $R$ is TP.
In other words, Theorem \ref{cw-c} follows from Theorem \ref{mmw-c}.
\end{rem}

\section*{Acknowledgement}

This work was supported in part by the National Natural Science Foundation of China
(Grant Nos. 11771065, 11701249).


\section*{References}

\begin{thebibliography}{99}
\bibitem{And87}T. Ando,
Totally positive matrices,
Linear Algebra Appl. 90 (1987) 165--219.
%\bibitem{Aig99}
%M. Aigner, Catalan-like numbers and determinants, J. Combin. Theory Ser. A 87 (1999) 33--51.
%\bibitem{Aig01}
%M. Aigner, Catalan and other numbers --- A recurrent theme,
%in: H. Crapo, D. Senato (Eds.), Algebraic Combinatorics and Computer Science, Springer, Berlin, 2001, 347--390.
%\bibitem{Bre89}
%F. Brenti, Unimodal, log-concave and P\'olya frequency sequences in combinatorics,
%Mem. Amer. Math. Soc. 413 (1989).
%\bibitem{Bre95}
%F. Brenti, Combinatorics and total positivity,
%J. Comb. Theory Ser. A 71 (1995) 175--218.
%\bibitem{Bre96}
%F. Brenti, The applications of total positivity to combinatorics, and conversely,
%in: Total Positivity and Its Applications,
%Jaca, 1994, in: Math. Appl., vol. 359, Kluwer, Dordrecht, 1996, pp. 451--473.
\bibitem{CLW15}X. Chen, H. Liang and Y. Wang,
Total positivity of Riordan arrays,
European J. Combin. 46 (2015) 68--74.
\bibitem{CLW15b}X. Chen, H. Liang and Y. Wang,
Total positivity of recursive matrices,
Linear Algebra Appl. 471 (2015) 383--393.
\bibitem{CW19}X. Chen and Y. Wang,
Notes on the total positivity of Riordan arrays,
Linear Algebra Appl. 569 (2019) 156--161.
%\bibitem{CK13}G.-S. Cheon and H. Kim,
%A new aspect of Hankel matrices via Krylov matrices,
%Linear Algebra Appl. 438 (2013) 361--373.
%\bibitem{CKS12}G.-S. Cheon, H. Kim and L.W. Shapiro,
%Combinatorics of Riordan arrays with identical $A$ and $Z$ sequences,
%Discrete Math.  312 (2012) 2040--2049.
%\bibitem{CS18}G.-S. Cheon and M. Song,
%A new aspect of Riordan arrays via Krylov matrices,
%Linear Algebra Appl. 554 (2018) 329--341.
\bibitem{FJ11}S.M. Fallat and C.R. Johnson,
Totally Nonnegative Matrices,
Princeton University Press, Princeton, NJ, 2011.
\bibitem{He15}
T.-X. He,
Matrix characterizations of Riordan arrays,
Linear Algebra Appl. 465 (2015) 15--42.
\bibitem{HS09}T.-X. He and R. Sprugnoli,
Sequence characterization of Riordan arrays,
Discrete Math. 309 (2009) 3962--3974.
\bibitem{Kar68}S. Karlin,
Total Positivity, Volume 1,
Stanford University Press, 1968.
\bibitem{Pin10}A. Pinkus,
Totally Positive Matrices,
Cambridge University Press, Cambridge, 2010.
\bibitem{SGWW91}L.W. Shapiro, S. Getu, W.-J. Woan and L.C. Woodson,
The Riordan group,
Discrete Appl. Math. 34 (1991) 229--239.
\bibitem{Slo20}R. S\l owik,
Some (counter) examples on totally positive Riordan arrays,
Linear Algebra Appl. 594 (2020) 117--123.
\bibitem{Spr94}R. Sprugnoli,
Riordan arrays and combinatorial sums,
Discrete Math. 132 (1994) 267--290.
%\bibitem{Spr11}R. Sprugnoli,
%Combinatorial sums through Riordan arrays,
%J. Geom.  101  (2011) 195--210.
\bibitem{WY18}Y. Wang and A.L.B. Yang,
Total positivity of Narayana matrices,
Discrete Math. 341 (2018) 1264--1269.
%\bibitem{WZ16}Y. Wang and B.-X. Zhu,
%Log-convex and Stieltjes moment sequences,
%Adv. in Appl. Math. 81 (2016) 115--127.
%\bibitem{Zhu17}B.-X. Zhu,
%Log-concavity and strong $q$-log-convexity for Riordan arrays and recursive matrices,
%Proc. Roy. Soc. Edinburgh Sect. A  147  (2017) 1297--1310.
\end{thebibliography}
\end{document}